\documentclass{article}
\usepackage{amsmath}
\usepackage{amsthm}
\usepackage{amsfonts}
\usepackage{latexsym}
\newtheorem{thm}{Theorem}[section]

\newtheorem{prop}[thm]{Proposition}

\newtheorem{lem}[thm]{Lemma}

\newcommand{\be}{\begin{equation}}
\newcommand{\ee}{\end{equation}}
\newcommand{\ben}{\begin{enumerate}}
\newcommand{\een}{\end{enumerate}}
\newcommand{\beq}{\begin{eqnarray}}
\newcommand{\eeq}{\end{eqnarray}}
\newcommand{\beqn}{\begin{eqnarray*}}
\newcommand{\eeqn}{\end{eqnarray*}}
\newcommand{\e}{\varepsilon}
\newcommand{\pa}{\partial}

\newcommand{\g}{{\bf g}}

\newcommand{\pxi}{ {\pa \over \pa x^i}}

\title{On the Second Approximate Matsumoto Metric}
\author{A. Tayebi, T. Tabatabaeifar and E. Peyghan}

\begin{document}   

\maketitle
\begin{abstract}
In this paper, we study the second approximate Matsumoto metric  $F=\alpha+\beta+\beta^2/\alpha+\beta^3/\alpha^2$  on a manifold $M$.  We prove that $F$ is of scalar flag curvature and isotropic  S-curvature if and only if it is isotropic Berwald metric with almost isotropic flag curvature.\\\\
{\bf {Keywords}}: Isotropic Berwald curvature, S-curvature, almost isotropic flag curvature.\footnote{ 2010 Mathematics subject Classification: 53C60, 53C25.}
\end{abstract}

\section{Introduction}
The flag curvature in Finsler geometry is a natural extension of the sectional curvature in Riemannian geometry, which is first introduced by L. Berwald. For a Finsler manifold $(M, F)$, the flag curvature is  a function ${\bf K}(P, y)$ of tangent planes $P\subset T_xM$ and  directions $y\in P$.  $F$  is said to be  of scalar flag curvature if the flag curvature ${\bf K}(P, y)={\bf K}(x, y)$ is independent of flags $P$ associated with any fixed flagpole $y$.  $F$ is called of almost isotropic flag curvature if
\be
{\bf K}=\frac{3c_{x^m}y^m}{F}+\sigma,\label{AIFC}
\ee
where $c=c(x)$ and $\sigma=\sigma(x)$ are   scalar functions  on $M$. One of the important problems in Finsler geometry is to characterize Finsler manifolds of almost isotropic flag curvature \cite{NST}.

To study the geometric properties of a Finsler metric, one also considers non-Riemannian quantities. In Finsler geometry,  there are several important non-Riemannian  quantities:  the  Cartan torsion ${\bf C}$,  the Berwald curvature ${\bf B}$, the  mean Landsberg curvature ${\bf J}$ and S-curvature ${\bf S}$, etc \cite{CS}\cite{LL}\cite{NST}\cite{TN}.  These are geometric quantities which vanish for Riemnnian metrics.

Among the non-Riemannian quantities, the S-curvature ${\bf S}={\bf S}(x, y)$ is closely related to the flag curvature
which constructed by Shen for given comparison theorems on Finsler manifolds. A Finsler metric $F$ is called of isotropic $S$-curvature if
\be
{\bf S}=(n+1)cF,\label{ISC}
\ee
for some scalar function $c=c(x)$ on $M$. In \cite{NST}, it is proved that if a
Finsler metric $F$ of scalar flag curvature is of isotropic S-curvature (\ref{ISC}), then it has almost isotropic flag curvature (\ref{AIFC}).

The geodesic curves of a Finsler metric $F=F(x,y)$ on a smooth manifold $M$, are determined  by $ \ddot c^i+2G^i(\dot c)=0$, where the local functions $G^i=G^i(x, y)$ are called the  spray coefficients.  A Finsler metric $F$ is  called a Berwald metric, if  $G^i$  are quadratic in $y\in T_xM$  for any $x\in M$.
A Finsler metric $F$ is said to be isotropic Berwald metric if its Berwald curvature  is in the following form
\be
B^i_{\ jkl}=c\Big\{F_{y^jy^k}\delta^i_{\ l}+F_{y^ky^l}\delta^i_{\ j}+F_{y^ly^j}\delta^i_{\ k}+F_{y^jy^ky^l}y^i\Big\},\label{IBC}\\
\ee
where $c=c(x)$ is a  scalar function  on $M$ \cite{CS}.

As a generalization of Berwald curvature, B\'{a}cs\'{o}-Matsumoto  proposed the notion of Douglas curvature \cite{BM}.  A Finsler metric is called a Douglas metric if  $G^i=\frac{1}{2}\Gamma^i_{jk}(x)y^jy^k + P(x,y)y^i$.

In order to find explicit examples of Douglas metrics, we consider $(\alpha, \beta)$-metrics. An $(\alpha, \beta)$-metric is a Finsler metric of the form $F:=\alpha \phi({\beta\over\alpha})$, where  $\phi=\phi(s)$ is a $C^\infty$ on $(-b_0, b_0)$ with certain regularity, $\alpha=\sqrt{a_{ij}(x)y^iy^j}$ is a Riemannian metric and $\beta =b_i(x)y^i$ is a 1-form on  $M$. This class of metrics is were first introduced by Matsumoto \cite{Mat5}. Among the $(\alpha, \beta)$-metrics, the  Matsumoto metric is special and significant metric which constitute a majority of actual research. The Matsumoto metric is expressed as
\[
F=\alpha\Big[1+\frac{\beta}{\alpha}+\big(\frac{\beta}{\alpha}\big)^2+\big(\frac{\beta}{\alpha}\big)^3+\cdots \Big].
\]
This metric was introduced by Matsumoto as a realization of Finsler's idea  ``a slope measure of a mountain with respect to a time measure" \cite{TPS}.  In the Matsumoto metric, the 1-form $\beta=b_iy^i$ was originally to be induced by earth gravity. Hence, we could regard $b_i(x)$ as the the infinitesimals and neglect the infinitesimals of degree of $b_i(x)$ more than two \cite{PC1}\cite{PC2}\cite{PC3}\cite{PLP}\cite{PLPK}. An approximate Matsumoto metric is a Finsler metric in the following form
\be
F=\alpha\Bigg[\sum_{k=0}^r (\frac{\beta}{\alpha})^k\Bigg],\label{a3}
\ee
where $|\beta|<|\alpha|$ (for more information, see \cite{PC2}). This metric was introduced by Park-Choi in \cite{PC2}. By definition, the Matsumoto metric is expressed as $\lim_{r\rightarrow \infty}L(\alpha,\beta)=\frac{\alpha^2}{\alpha - \beta}$.

In this paper, we consider second  approximate Matsumoto metric $F=\alpha+\beta+\frac{\beta^2}{\alpha}+\frac{\beta^3}{\alpha^2}$ with some non-Riemannian curvature properties and  prove the following.

\begin{thm}\label{MAINTHM}
Let $F=\alpha+\beta+\frac{\beta^2}{\alpha}+\frac{\beta^3}{\alpha^2}$ be a non-Riemannian second approximate Matsumoto metric on a manifold $M$ of dimension $n$. Then $F$ is of scaler flag curvature with isotropic $S$-curvature (\ref{ISC}), if and only if it has isotropic Berwald curvature (\ref{IBC}) with almost isotropic flag curvature (\ref{AIFC}). In this case, $F$ must be locally Minkowskian.
\end{thm}

\section{Preliminaries}
Let $M$ be a n-dimensional $ C^\infty$ manifold. Denote by $T_x M $ the tangent space at $x \in M$,  by $TM=\cup _{x \in M} T_x M $ the tangent bundle of $M$, and by $TM_{0} = TM \setminus \{ 0 \}$ the slit tangent bundle on $M$. A  Finsler metric on $M$ is a function $ F:TM \rightarrow [0,\infty)$ which has the following properties:\\
(i) $F$ is $C^\infty$ on $TM_{0}$;\\
(ii) $F$ is positively 1-homogeneous on the fibers of tangent bundle $TM$;\\
(iii) for each $y\in T_xM$, the following quadratic form ${\bf g}_y$ on
$T_xM$  is positive definite,
\[
{\bf g}_{y}(u,v):={1 \over 2} \frac{\partial^2}{\partial s \partial t}\left[  F^2 (y+su+tv)\right]|_{s,t=0}, \ \
u,v\in T_xM.
\]
Let  $x\in M$ and $F_x:=F|_{T_xM}$.  To measure the non-Euclidean feature of $F_x$, define ${\bf C}_y:T_xM\otimes T_xM\otimes T_xM\rightarrow \mathbb{R}$ by
\[
{\bf C}_{y}(u,v,w):={1 \over 2} \frac{d}{dt}\Big[{\bf g}_{y+tw}(u,v)
\Big]|_{t=0}, \ \ u,v,w\in T_xM.
\]
The family ${\bf C}:=\{{\bf C}_y\}_{y\in TM_0}$  is called the Cartan torsion. It is well known that ${\bf{C}}=0$ if and only if $F$ is Riemannian \cite{ShDiff}. For  $y \in T_xM_0$, define  mean Cartan torsion ${\bf I}_y$ by ${\bf I}_y(u) := I_i(y)u^i$, where $I_i:=g^{jk}C_{ijk}$. By Diecke Theorem, $F$ is Riemannian if and only if ${\bf I}_y=0$.

The horizontal covariant derivatives of ${\bf I}$ along geodesics give rise to  the mean Landsberg curvature ${\bf J}_y(u):=J_i (y)u^i$, where $J_i:=I_{i|s}y^s$. A Finsler metric is said to be  weakly Landsbergian if ${\bf J}=0$.

\bigskip

Given a Finsler manifold $(M,F)$, then a global vector field ${\bf G}$ is induced by $F$ on $TM_0$, which in a standard coordinate $(x^i,y^i)$ for $TM_0$ is given by ${\bf G}=y^i {{\partial} \over {\partial x^i}}-2G^i(x,y){{\partial} \over {\partial y^i}}$, where
\[
G^i:=\frac{1}{4}g^{il}\Big[\frac{\partial^2(F^2)}{\partial x^k \partial y^l}y^k-\frac{\partial(F^2)}{\partial x^l}\Big],\ \ \ \  y\in T_xM.
\]
The ${\bf G}$ is called the  spray associated  to $(M,F)$.  In local coordinates, a curve $c(t)$ is a geodesic if and only if its coordinates $(c^i(t))$ satisfy $ \ddot c^i+2G^i(\dot c)=0$.

\bigskip

For a tangent vector $y \in T_xM_0$, define ${\bf B}_y:T_xM\otimes T_xM \otimes T_xM\rightarrow T_xM$ and ${\bf E}_y:T_xM \otimes T_xM\rightarrow \mathbb{R}$ by ${\bf B}_y(u, v, w):=B^i_{\ jkl}(y)u^jv^kw^l{{\partial } \over {\partial x^i}}|_x$ and ${\bf E}_y(u,v):=E_{jk}(y)u^jv^k$
where
\[
B^i_{\ jkl}:={{\partial^3 G^i} \over {\partial y^j \partial y^k \partial y^l}},\ \ \ E_{jk}:={{1}\over{2}}B^m_{\ jkm}.
\]
The $\bf B$ and $\bf E$ are called the Berwald curvature and mean Berwald curvature, respectively.  Then $F$ is called a Berwald metric and weakly Berwald metric if $\bf{B}=0$ and $\bf{E}=0$, respectively.

A Finsler metric $F$ is said to be isotropic mean Berwald metric if its  mean Berwald curvature is in the following form
\be
E_{ij}=\frac{n+1}{2F}ch_{ij},\label{IMBC}
\ee
where $c=c(x)$ is a  scalar function  on $M$ and $h_{ij}$ is the angular metric \cite{CS}.

\bigskip

Define ${\bf D}_y:T_xM\otimes T_xM \otimes T_xM\rightarrow T_xM$  by
${\bf D}_y(u,v,w):=D^i_{\ jkl}(y)u^iv^jw^k\frac{\partial}{\partial x^i}|_{x}$ where
\[
D^i_{\ jkl}:=B^i_{\ jkl}-{2\over
n+1}\{E_{jk}\delta^i_l+E_{jl}\delta^i_k+E_{kl}\delta^i_j+E_{jk,l}
y^i\}.
\]
We call ${\bf D}:=\{{\bf D}_y\}_{y\in TM_{0}}$ the Douglas curvature. A Finsler metric with ${\bf D}=0$ is called a Douglas metric. The notion of Douglas metrics was proposed by B$\acute{a}$cs$\acute{o}$-Matsumoto as a generalization  of Berwald metrics \cite{BM}.

For a Finsler metric $F$ on an $n$-dimensional manifold $M$, the
Busemann-Hausdorff volume form $dV_F = \sigma_F(x) dx^1 \cdots
dx^n$ is defined by
$$
\sigma_F(x) := {{\rm Vol} (\Bbb B^n(1))
\over {\rm Vol} \Big \{ (y^i)\in R^n \ \Big | \ F \Big ( y^i
\pxi|_x \Big ) < 1 \Big \} } .\label{dV}
$$
In general, the local scalar function $\sigma_F(x)$ can not be
expressed in terms of elementary functions, even $F$ is locally
expressed by elementary functions. Let $G^i$ denote the geodesic coefficients of $F$ in the same
local coordinate system. The S-curvature can be defined by
$$
 {\bf S}({\bf y}) := {\pa G^i\over \pa y^i}(x,y) - y^i {\pa \over \pa x^i}
\Big [ \ln \sigma_F (x)\Big ],\label{Slocal}
$$
where ${\bf y}=y^i\pxi|_x\in T_xM$. It is proved that ${\bf S}=0$ if $F$ is a Berwald metric. There are many non-Berwald metrics satisfying ${\bf S}=0$. ${\bf S}$ said to be  {\it isotropic} if there is a scalar functions $c(x)$ on $M$ such that ${\bf S}=(n+1)c(x) F$.

\bigskip

The Riemann curvature ${\bf R}_y= R^i_{\ k}  dx^k \otimes \pxi|_x :
T_xM \to T_xM$ is a family of linear maps on tangent spaces, defined
by
\[
R^i_{\ k} = 2 {\pa G^i\over \pa x^k}-y^j{\pa^2 G^i\over \pa
x^j\pa y^k} +2G^j {\pa^2 G^i \over \pa y^j \pa y^k} - {\pa G^i \over
\pa y^j} {\pa G^j \over \pa y^k}.\label{Riemann}
\]
For a flag $P={\rm span}\{y, u\} \subset T_xM$ with flagpole $y$, the  flag curvature ${\bf
K}={\bf K}(P, y)$ is defined by
\[
{\bf K}(P, y):= {\g_y (u, {\bf R}_y(u)) \over \g_y(y, y) \g_y(u,u)
-\g_y(y, u)^2 }.
\]
We say that a Finsler metric $F$ is   of scalar curvature if for any $y\in
T_xM$, the flag curvature ${\bf K}= {\bf K}(x, y)$ is a scalar
function on the slit tangent bundle $TM_0$. In this case, for some scalar function ${\bf K}$ on $TM_0$ the Riemann curvature is in the following form
\[
R^i_{\ k}= {\bf K}F^2 \{ \delta^i_k - F^{-1}F_{y^k} y^i \}.
\]
If ${\bf K}=constant$, then $F$ is said to be of  constant flag curvature.  A Finsler metric $F$ is called {\it isotropic flag curvature}, if ${\bf K}= {\bf K}(x)$.

\section{Proof of Theorem \ref{MAINTHM}}
Let $F=\alpha\phi(s)$, $s=\frac{\beta}{\alpha}$ be an $(\alpha, \beta)$-metric,  where  $\phi=\phi(s)$ is a $C^\infty$ on $(-b_0, b_0)$ with certain regularity, $\alpha =\sqrt{a_{ij}(x)y^iy^j}$ is a Riemannian metric and  $\beta = b_i(x)y^i$  is a 1-form on a manifold $M$. Let
\[
r_{ij}:= {1\over 2} \Big[b_{i|j}+b_{j|i} \Big], \ \ \ \
s_{ij} := {1\over 2} \Big[ b_{i|j} - b_{j|i} \Big].
\]
\[
r_j := b^i r_{ij}, \ \ \ \ s_j:=b^i s_{ij}.
\]
where $b_{i|j}$ denote the coefficients of the covariant derivative of $\beta$ with respect to $\alpha$. Let
\[
r_{i0}: = r_{ij}y^j, \ \  s_{i0}:= s_{ij}y^j, \ \ r_0:= r_j y^j, \ \ s_0 := s_j y^j.
\]
Put
\begin{align}
\nonumber Q&=\frac{\phi'}{\phi- s\phi},\\
\nonumber \Theta &=\frac{\phi \phi' -s(\phi \phi'' +\phi'^2)}{2\phi\big[(\phi -s\phi')+(b^2-s^2)\phi''\big]}\\
\Psi &=\frac{\phi''}{2\big[(\phi -s\phi')+(b^2-s^2)\phi''\big]}.\label{a2}
\end{align}
Then the $S$-curvature is given by
\begin{eqnarray}
\nonumber {\bf S}=\Big[Q'-2\Psi Qs\!\!\!\!&-&\!\!\!\!\ 2(\Psi Q)'(b^2-s^2)-2(n+1)Q\Theta +2\lambda \Big]s_0\\
\!\!\!\!&+&\!\!\!\!\ 2(\Psi +\lambda)s_0+\alpha^{-1}\Big[(b^2-s^2)\Psi' +(n+1)\Theta \Big]r_{00}.\label{a3}
\end{eqnarray}
Let us put
\begin{eqnarray}
\nonumber \Delta\!\!\!\!&:=&\!\!\!\!\ 1+sQ+(b^2-s^2)Q',\\
\nonumber \Phi \!\!\!\!&:=&\!\!\!\!\ -(n\Delta +1+sQ)(Q-sQ')-(b^2-s^2)(1 + sQ)Q''.
\end{eqnarray}
In \cite{ChSh3}, Cheng-Shen  characterize $(\alpha,\beta)$-metrics with isotropic
S-curvature.
\begin{lem} \label{ChSh}{\rm(\cite{ChSh3})} \emph{Let  $ F =\alpha \phi (\beta/\alpha)$ be an $(\alpha,\beta)$-metric on an $n$-manifold. Then, $F$ is of isotropic S-curvature $ {\bf S} = (n+1) cF$, if and only if one of the following holds
\item[(i)] $\beta $ satisfies
\be
r_{ij} = \e \Big \{ b^2 a_{ij} -b_i b_j \Big \}, \ \ \ \ \ s_j=0, \label{CS1}
\ee
where $\e= \e(x)$ is a scalar function, and $\phi=\phi(s)$ satisfies
\be
 \Phi = -2 (n+1) k \frac{ \phi \Delta^2}{b^2-s^2}, \label{CS}
\ee
where $k$ is a constant. In this case, $c=k\epsilon$.
\item[(ii)] $\beta$ satisfies
\be
r_{ij} =0,\ \ \ \ \ s_j =0.\label{CS2}
\ee
In this case, $ c =0$.}
\end{lem}

\bigskip

Let
\begin{eqnarray}
\nonumber \Psi_1 \!\!\!\!&:=&\!\!\!\!\ \sqrt{b^2-s^2}\Delta^{\frac{1}{2}}\Big[\frac{\sqrt{b^2-s^2}\Phi}{\Delta^{\frac{3}{2}}}\Big]',\\
\nonumber \Psi_2\!\!\!\!&:=&\!\!\!\!\ 2(n+1)(Q-sQ')+3\frac{\Phi}{\Delta},\\
\theta \!\!\!\!&:=&\!\!\!\!\ \frac{Q-sQ'}{2\Delta}.\label{b1}
\end{eqnarray}
Then the formula for the mean Cartan torsion of an $(\alpha,\beta)$-metric is given by following
\begin{align}
\nonumber I_i&=\frac{1}{2}\frac{\partial }{\partial y^i}\Big[(n+1)\frac{\phi'}{\phi}-(n-2)\frac{s\phi''}{\phi-s\phi'}-\frac{3s\phi''-(b^2-s^2)\phi'''}
{(\phi-s\phi')+(b^2-s^2)\phi''}\Big]\\
&=-\frac{\Phi(\phi- s\phi')}{2\Delta \phi \alpha^2}(\alpha b_i-sy_i).\label{bb}
\end{align}
In \cite{ChWW}, it is proved that the condition $\Phi=0$ characterizes the Riemannian metrics among $(\alpha,\beta)$-metrics. Hence, in the continue, we suppose that $\Phi\neq 0$.

\bigskip

Let $G^i=G^i(x,y)$ and $\bar{G}^i_{\alpha}=\bar{G}^i_{\alpha}(x,y)$ denote the coefficients  of $F$ and  $\alpha$ respectively in the same coordinate system. By definition, we have
\be
G^i = \bar{G}^i_{\alpha} + P y^i + Q^i,\label{spray}
\ee
where
\begin{eqnarray*}
P \!\!\!\!&:=&\!\!\!\!\  \alpha^{-1} \Theta \Big[ -2 Q \alpha  s_0 +  r_{00}\Big ]\\
Q^i \!\!\!\!&:=&\!\!\!\!\ \alpha Q s^i_{\ 0} + \Psi \Big [-2 Q\alpha  s_0 +
r_{00}   \Big ]  b^i.
\end{eqnarray*}
Simplifying (\ref{spray}) yields the following
\begin{equation}\label{b2}
G^i=\bar{G}^i_{\alpha}+\alpha Qs^i_0 +\theta (-2\alpha Qs_0+r_{00})\Big[\frac{y^i}{\alpha}+\frac{Q'}{Q-sQ'}b^i\Big].
\end{equation}
Clearly, if $\beta $ is parallel with respect to $\alpha$ ($r_{ij}=0$ and $s_{ij}=0$), then $P=0$ and  $Q^i=0$. In this case, $G^i =\bar{G}^i_{\alpha}$ are quadratic in $y$, and $F$ is a Berwald metric.

\bigskip

For an $(\alpha,\beta)$-metric $F=\alpha \phi(s)$, the mean Landsberg curvature is given by
\begin{eqnarray}
\nonumber J_i=\!\!\!\!&-&\!\!\!\!\!\!\!\ \frac{1}{2\Delta \alpha^4}\Bigg[\frac{2\alpha^2}{b^2-s^2}\big[\frac{\Phi}{\Delta}+(n+1)(Q-sQ')\big](r_0+s_0)h_i\\
\nonumber \!\!\!\!&+&\!\!\!\!\!\!\!\ \frac{\alpha}{b^2-s^2}(\Psi_1 +s\frac{\Phi}{\Delta})(r_{00}-2\alpha Qs_0)h_i +\alpha\Big[-\alpha Q's_0h_i+\alpha Q(\alpha^2 s_i-y_is_0)\\
\!\!\!\!&+&\!\!\!\!\!\!\!\ \alpha^2 \Delta s_{i0}+\alpha^2(r_{i0}-2\alpha Qs_i)-(r_{00}-2\alpha Qs_0)y_i\Big]\frac{\Phi}{\Delta}\Bigg].\label{b3}
\end{eqnarray}
Contracting (\ref{b3}) with $b^i=a^{im}b_m$ yields
\begin{eqnarray}
\bar{J}:=J_ib^i=-\frac{1}{2\Delta \alpha^2}\Big[\Psi_1(r_{00}-2\alpha Qs_0)+\alpha \Psi_2(r_0+s_0)\Big].\label{b4}
\end{eqnarray}
The horizontal covariant derivatives $J_{i;m}$ and $J_{i|m}$ of $J_i$ with respect to $F$ and $\alpha$, respectively,  are given by
\begin{align*}
J_{i;m}&=\frac{\partial J_i}{\partial x^m}-J_l\Gamma^l_{im}-\frac{\partial J_i}{\partial y^l}N^l_m, \\      J_{i|m}&=\frac{\partial J_i}{\partial x^m}-J_l\bar{\Gamma}^l_{im}-\frac{\partial J_i}{\partial y^l}\bar{N}^l_m.
\end{align*}
Then  we have
\begin{equation}
J_{i;m}y^m=J_{i|m}y^m-J_l(N^l_i-\bar{N}_i^l)-2\frac{\partial J_i}{\partial y^l}(G^l-\bar{G}^l).\label{o0}
\end{equation}
Let $F$ be a Finsler metric of scalar flag curvature ${\bf K}$. By Akbar-Zadeh's theorem it satisfies following
\be
A_{ijk;s;m}y^sy^m+{\bf K}F^2 A_{ijk}+\frac{F^2}{3}\Big[h_{ij}{\bf K}_k+h_{jk}{\bf K}_j+h_{ki}{\bf K}_j\Big]=0,\label{Akbar1}
\ee
where $A_{ijk}=FC_{ijk}$ is the Cartan torsion and ${\bf K}_i=\frac{\partial {\bf K}}{\partial y^i}$ \cite{BCS}. Contracting (\ref{Akbar1}) with $g^{ij}$ yields
\be
J_{i;m}y^m+{\bf K}F^2I_i+\frac{n+1}{3}F^2{\bf K}_i=0.\label{Akbar2}
\ee
By (\ref{o0}) and (\ref{Akbar2}), for an $(\alpha,\beta)$-metric $F=\alpha \phi(s)$ of constant flag curvature ${\bf K}$, the following holds
\begin{equation}
J_{i|m}-J_l\frac{\partial(G^l-\bar{G}^l)}{\partial y^i}b^i-2\frac{\partial \bar{J}}{\partial y^l}(G^l-\bar{G}^l){\bf K}\alpha^2\phi^2 I_i=0.\label{o1}
\end{equation}
Contracting (\ref{o1}) with  $b^i$ implies that
\begin{equation}\label{b5}
\bar{J}_{|m}y^m-J_ia^{ik}b_{k|m}y^m-J_l\frac{\partial(G^l-\bar{G}^l)}{\partial y^i}b^i-2\frac{\partial \bar{J}}{\partial y^l}(G^l-\bar{G}^l)+\boldsymbol{K}\alpha^2 \phi^2 I_ib^i=0.
\end{equation}

\bigskip

There exists a relation between mean Berwald curvature ${\bf E}$ and the S-curvature ${\bf S}$. Indeed, taking twice vertical covariant derivatives  of the S-curvature gives rise the $E$-curvature.   It is easy to see that,  every Finsler metric of isotropic S-curvature (\ref{ISC}) is of isotropic mean Berwald curvature (\ref{IMBC}).  Now, is the equation ${\bf S}= (n+1) c F$ equivalent to the equation ${\bf E}= \frac{n+1}{2} c F^{-1}{\bf h}$?

Recently, Cheng-Shen prove that a Randers metric $F=\alpha+\beta$ is of isotropic $S$-curvature if and only if it is of isotropic $E$-curvature \cite{ChSh2}. Then,  Chun-Huan-Cheng extend this equivalency to the Finsler metric  $F=\alpha^{-m}(\alpha+\beta)^{m+1}$ for every real constant $m$,  including Randers metric \cite{CC}.  In \cite{Cui}, Cui extend their result and show that for the Matsumoto metric $F=\frac{\alpha^2}{\alpha-\beta}$ and the special  $(\alpha, \beta)$-metric $F=\alpha+\epsilon\beta+\kappa(\beta^2/\alpha)$ $(\kappa\neq 0)$, these notions are equivalent.

\bigskip

To prove Theorem \ref{MAINTHM}, we need the following.

\begin{prop}\label{prop1}
Let  $F=\alpha+\beta+\frac{\beta^2}{\alpha}+\frac{\beta^3}{\alpha^2}$ be a second approximate Matsumoto metric on a  manifold $M$ of dimension $n$. Then  the following are equivalent
\begin{description}
\item[(i)]  $F$ has isotropic $S$-curvature, ${\bf S}=(n+1)c(x)F$;
\item[(ii)]  $F$ has isotropic mean Berwald curvature, ${\bf E}=\frac{n+1}{2}c(x)F^{-1}{\bf h}$;
\end{description}
where $c=c(x)$ is a scalar function on the manifold $M$. In this case, ${\bf S}=0$. Then $\beta$ is  a Killing $1$-form with constant length with respect to $\alpha$, that is, $r_{00}=0$.
\end{prop}
\begin{proof}
$(i) \rightarrow (ii)$ is obvious. Conversely,  suppose that $F$ has isotropic mean Berwald curvature, $\textbf{E}=\frac{(n+1)}{2}c(x)F^{-1}\textbf{h}$. Then we have
\begin{equation}\label{a4}
{\bf S}=(n+1)[cF+\eta],
\end{equation}
where $\eta=\eta_i(x)y^i$ is a 1-form on $M$. For the second approximate Matsumoto metric, (\ref{a2}) reduces to following
\begin{eqnarray}
\nonumber Q\!\!\!\!&=&\!\!\!\!\!\!\!\ -\frac{1+2s+3s^2}{-1+s^2+2s^3},\\
\nonumber \Theta \!\!\!\!&=&\!\!\!\!\!\!\!\ \frac{1}{2}\frac{1-6s^2-12s^3-15s^4-12s^5}{(1+s+s^2+s^3)(1-3s^2-8s^3+2b^2+6b^2s)},\\
\Psi \!\!\!\!&=&\!\!\!\!\!\!\!\ \frac{1+3s}{(1-3s^2-8s^3+2b^2+6b^2s)}.\label{a5}
\end{eqnarray}
By substituting (\ref{a4}) and (\ref{a5}) in (\ref{a3}), we have
\begin{eqnarray}
\nonumber {\bf S}=\!\!\!\!&&\!\!\!\!\!\!\!\!\!\ \Bigg[\frac{2(1+3s)(1+s+s^2+s^3)}{(-1+s^2+2s^3)^2}+\frac{2(1+3s)(1+2s+3s^2)s}{(1-3s^2-8s^3+2b^2+6b^2s)(-1+s^2+2s^3)}\\
\nonumber \!\!\!\!&-&\!\!\!\!\!\!\!\ \frac{2(5+26s+77s^2+88s^3-61s^4-430s^5-805s^6+4b^2+40b^2s+148b^2s^2)
(b^2-s^2)}{(1-3s^2-8s^3+2b^2+6b^2s)^2(-1+s^2+2s^3)^2}\\
\nonumber\!\!\!\!&-&\!\!\!\!\!\!\!\  \frac{2(256s^3b^2+252s^4b^2+216s^5b^2+108s^6b^2-828s^7-432s^8)(b^2-s^2)}
{(1-3s^2-8s^3+2b^2+6b^2s)^2(-1+s^2+2s^3)^2}\\
\nonumber \!\!\!\!&+&\!\!\!\!\!\!\!\ \frac{(n+1)(1+2s+3s^2)(1-6s^2-12s^3-15s^4-12s^5)}{(-1+s^2+2s^3)(1+s+s^2+s^3)(1-3s^2-8s^3+2b^2+6b^2s)}+2\lambda \Bigg]s_0\\
\nonumber \!\!\!\!&+&\!\!\!\!\!\!\!\ 2\bigg[\frac{(1+3s)}{1-3s^2-8s^3+2b^2+6b^2s}+\lambda \bigg]+\bigg[\frac{3(b^2-s^2)(1+11s^2+16s^3+2s)}{\alpha(1-3s^2-8s^3+2b^2+6b^2s)^2}\bigg]r_{00}\\
\nonumber \!\!\!\!&+&\!\!\!\!\!\!\!\ \bigg[\frac{(n+1)(1-6s^2-12s^3-15s^4-12s^5)}{2\alpha(1+s+s^2+s^3)(1-3s^2-8s^3+2b^2+6b^2s)}\bigg]r_{00}\\
=\!\!\!\!&&\!\!\!\!\!\!\!\!\!\!\ (n+1)\big[c\alpha(1+s+s^2+s^3)+\eta\big].\label{a6}
\end{eqnarray}
Multiplying (\ref{a6})  with  $(-1+s^2+2s^3)(1+s+s^2+s^3)(1-3s^2-8s^3+2b^2+6b^2s)^2\alpha^{14}$ implies that
\begin{eqnarray}
\nonumber M_1 + M_2\alpha^2 +M_3\alpha^4 +M_4\alpha^6 +M_5 \alpha^8 +M_6\alpha^{10}+M_7\alpha^{12}+M_8\alpha^{14}\\
\nonumber \ \ \ \ \  \ \ \ \  \ + \alpha \Big[M_9 +M_{10}\alpha^2 +M_{11}\alpha^4 +M_{12}\alpha^6 +M_{13}\alpha^8 +M_{14}\alpha^{10}\\
 +M_{15}\alpha^{12}+M_{16}\alpha^{14}\Big]=0,\label{a7}
\end{eqnarray}
where
\begin{eqnarray*}
M_1:=\!\!\!\!&&\!\!\!\!\!\!\!\!\!\ -128(n+1)c\beta^{15},\\
M_2:=\!\!\!\!&&\!\!\!\!\!\!\!\!\!\ 2\Big[(n+1)\big[(-385+96b^2)c\beta^2 -64\eta \beta\big]+128\lambda(r_0+s_0)\beta +48nr_{00}\Big]\beta^{11},\\
M_3:=\!\!\!\!&&\!\!\!\!\!\!\!\!\!\ -\Big[(n+1)[4(-283b^2+243+18b^4)c\beta^2 -6(32b^2-59)\eta \beta]\\
\!\!\!\!&+&\!\!\!\!\!\!\ [12(-59+32b^2)\lambda(r_0+s_0)-96((3n-1)s_0-r_0]\beta \\
\!\!\!\!&+&\!\!\!\!\!\!\ 3(-24b^2+36-219n+72nb^2)r_{00} \Big]\beta^9,\\
M_4:=\!\!\!\!&&\!\!\!\!\!\!\!\!\!\ \Big[(n+1)\big[-2(29-738b^2+208b^4)c\beta^2 -3(-172b^2+21+24b^4)\eta \beta\big]\\
\!\!\!\!&+&\!\!\!\!\!\!\ (-159nb^2-132-39n+96b^2)r_{00}-3(-48b^2+123+72nb^2-277n)s_0\beta \\
\!\!\!\!&+&\!\!\!\!\!\!\ 3(24b^2+86)r_0\beta+6\lambda(-172b^2+21+24b^4)(s_0+r_0)\beta \Big]\beta^7,\\
M_5:=\!\!\!\!&&\!\!\!\!\!\!\!\!\!\ \Big[(n+1)[-8(-27b^2+70b^4-41)c\beta^2 -4(47b^4-40-32b^2)\eta \beta]\\
\!\!\!\!&+&\!\!\!\!\!\!\ (15nb^2+20-55n+108b^2)r_{00}+8\lambda(47b^4-40-32b^2)(r_0+s_0)\beta\\
\!\!\!\!&-&\!\!\!\!\!\!\ 2\big[(-262b^2+278+303nb^2-147n)s_0-(94b^2+32)r_0\big]\beta \Big]\beta^5
\\
M_6:=\!\!\!\!&&\!\!\!\!\!\!\!\!\!\ 2\Big[(n+1)\big[-2(4b+1)(4b-1)(4b^2+13)c\beta^2 -10(1+20b^2+6b^4)\eta\beta\big]\\
\!\!\!\!&+&\!\!\!\!\!\!\ (11n-40b^2+35nb^2+20)r_{00}+20\lambda (1+20b^2+6b^4)(r_0+s_0)\beta\\
\!\!\!\!&-&\!\!\!\!\!\!\ 2(32n+51+129nb^2-294b^2)s_0\beta-(30b^2-50)r_0\beta \Big]\beta^3,\\
M_7:=\!\!\!\!&&\!\!\!\!\!\!\!\!\!\ -\Big[(n+1)\big[-4(-17b^2-7+30b^4)c\beta^2 -6(-1+10b^4)\beta\big]+(3n+12)b^2r_{00}\\
\!\!\!\!&-&\!\!\!\!\!\!\ 12\lambda(1-10b^4)(s_0+r_0) -6\big[(nb^2-n-2+14b^2)s_0-10b^2r_0\big]\beta \Big]\beta,\\
M_8:=\!\!\!\!&&\!\!\!\!\!\!\!\!\!\ 4\Big[(n+1)\Big[2(1+2b^2)(8b^2+1)c\beta +(1+2b^2)^2\eta\Big ]-2\lambda(1+2b^2)^2(s_0+r_0)\\
\!\!\!\!&+&\!\!\!\!\!\!\ \big[-57nb^2-64(n+1)b^4-8n-55b^2+6\big]r_{00}-(2-4b^2)r_0\\
\!\!\!\!&+&\!\!\!\!\!\!\  \big[n+(4+2n)b^2-1\big]s_0\Big],\\
M_9:=\!\!\!\!&&\!\!\!\!\!\!\!\!\!\ -416(n+1)c\beta^{14},\\
M_{10}:=\!\!\!\!&&\!\!\!\!\!\!\!\!\!\  \Big[(n+1)[(-1037+616b^2)c\beta^2 -288\eta \beta]+576\lambda(r_0+s_0)\beta +(204n-6)r_{00} \Big]\beta^{12},
\\
M_{11}:=\!\!\!\!&&\!\!\!\!\!\!\!\!\!\ -\frac{1}{2}\Big[ (n+1)[8(57b^4-385b^2+143)c\beta^2 -2(424b^2-267)\eta \beta]\\
\!\!\!\!&+&\!\!\!\!\!\!\  \big[4\lambda(424b^2-267)(r_0+s_0)-4(-115+330n)s_0+1320r_0\big]\beta\\
\!\!\!\!&+&\!\!\!\!\!\!\  (300nb^2+249-189n-120b^2)r_{00} \Big]\beta^8,
\\
M_{12}:=\!\!\!\!&&\!\!\!\!\!\!\!\!\!\ 4\Big[(n+1)[-(572b^4-932b^2-275)c\beta^2 -4(39b^4-28-102b^2)\eta \beta]\\
\!\!\!\!&-&\!\!\!\!\!\!\ 75nb^2-62-89n+144b^2r_{00}+8\lambda(39b^4-28-102b^2)(r_0+s_0)\beta\\
\!\!\!\!&-&\!\!\!\!\!\!\ 6[(-58b^2+100+81nb^2-109n)s_0-(26b^2+34)r_0]\beta \Big]\beta^6,
\end{eqnarray*}
\begin{eqnarray*}
M_{13}:=\!\!\!\!&&\!\!\!\!\!\!\!\!\!\ \Big[(n+1)[-8(-24+35b^2+47b^4)c\beta^2 -6(26b^4-11+20b^2)\eta \beta]\\
\!\!\!\!&+&\!\!\!\!\!\!\  (51nb^2+33-6n+24b^2)r_{00}+12\lambda(26b^4-11+20b^2)(s_0+r_0)\\
\!\!\!\!&-&\!\!\!\!\!\!\  6(-118b^2+54+83nb^2-3n)s_0\beta+6(26b^2-10)r_0\beta\Big]\beta^4,
\\
M_{14}:=\!\!\!\!&&\!\!\!\!\!\!\!\!\!\ -\Big[(n+1)[-(56b^4-272b^2-39)c\beta^2 -4(7(n+1)b^4-6(1+n)-22nb^2)\eta \beta]\\
\!\!\!\!&+&\!\!\!\!\!\!\ (-7nb^2-8-5n+32b^2)r_{00}+8\lambda(7b^4-6-22b^2)(r_0+s_0)\beta\\
\!\!\!\!&+&\!\!\!\!\!\!\ 2(-154b^2+12+33nb^2+19n)s_0\beta+2(14b^2-22)s_0\beta \Big]\beta^2\\
M_{15}:=\!\!\!\!&&\!\!\!\!\!\!\!\!\!\ \Big[(n+1)[4(23b^4+5b^2-1)c\beta^2 -(14b^2+1)(2b^2+1)\eta \beta]\\
\!\!\!\!&-&\!\!\!\!\!\!\ \frac{n+1}{2}(8b^2+1)r_{00}-2\lambda(14b^2+1)(2b^2+1)(s_0+r_0)\beta\\
\!\!\!\!&-&\!\!\!\!\!\!\ 2(-6b^2+3-5nb^2)s_0\beta-2(4+14b^2)r_0\beta \Big],\\
M_{16}:=\!\!\!\!&&\!\!\!\!\!\!\!\!\!\ (n+1)(1+2b^2)^2c
\end{eqnarray*}
The term of (\ref{a7}) which is seemingly does not contain $\alpha^2$ is $M_1$. Since $\beta^{15}$ is not divisible by $\alpha^2$, then  $c=0$ which implies that
\[
M_1=M_9=0.
\]
Therefore (\ref{a7}) reduces to following
\begin{align}
&M_2+M_3\alpha^2 +M_4\alpha^4 +M_5 \alpha^6 +M_6\alpha^{8}+M_7\alpha^{10}+M_8\alpha^{12}=0,\\
&M_{10}+M_{11}\alpha^2 +M_{12}\alpha^4 +M_{13}\alpha^6 +M_{14}\alpha^{8} +M_{15}\alpha^{10}+M_{16}\alpha^{12}=0.\label{a8}
\end{align}
By plugging $c=0$ in $M_2$ and $M_{10}$, the only equations that don't contain $\alpha^2$ are the following
\begin{align}
8\Big[8(2\lambda(r_0+s_0)-(n+1)\eta )+6nr_{00}\Big]&=\tau_1 \alpha^2,\label{a9}\\
6\Big[48(2\lambda(r_0+s_0)-(n+1)\eta) +(34n-1)r_{00}\Big]&=\tau_2 \alpha^2, \label{a10}
\end{align}
where $\tau_1=\tau_1(x)$ and $\tau_2=\tau_2(x)$ are scalar functions on $M$. By eliminating $[2\lambda(r_0+s_0)-(n+1)\eta]$ from (\ref{a9}) and (\ref{a10}), we get
\begin{equation}\label{a11}
r_{00}=\tau \alpha^2,
\end{equation}
where $\tau=\frac{\tau_2 -\tau_1}{-(18n+1)}$. By (\ref{a9}) or (\ref{a10}), it follows that
\begin{equation}\label{a12}
2\lambda(r_0+s_0)-(n+1)\eta=0.
\end{equation}
By (\ref{a11}), we have $r_0=\tau \beta$. Putting  (\ref{a11}) and (\ref{a12}) in $M_{10}$ and $M_{11}$ yield
\begin{eqnarray}
M_{10}\!\!\!\!\!\!&=&\!\!\!\!\!\!\!(204n-6)\tau \alpha^2 \beta^{12},\label{m10}\\
M_{11}\!\!\!\!\!\!\!&=&\!\!\!\!\!\!\!\Big[[(660n-230)s_0-660r_0]\beta - \frac{(300n-120)b^2+249-189n}{2}r_{00}\tau \alpha^2\Big]\beta^9.
\label{m11}
\end{eqnarray}
By putting (\ref{m10}) and (\ref{m11}) into (\ref{a8}), we have
\begin{eqnarray}
\nonumber \!\!\!\!\!\!&&\!\!\!\!\![(660n-230)s_0-660r_0]\beta^{10}-\frac{300nb^2+249-189n-120b^2}{2}r_{00}\tau \alpha^2\beta^9 \\
\!\!\!\!\!\!&+&\!\!\!\!\! (204n-6)\tau \beta^{12}-M_{12}\alpha^2 +M_{13}\alpha^4 +M_{14}\alpha^6+M_{15}\alpha^8+M_{16}\alpha^{10}=0.\label{m12}
\end{eqnarray}
The only equations of (\ref{m12}) that do not contain $\alpha^2$ is $[(204n-6)\tau \beta^2+(660n-230)s_0-660r_0]\beta^{10}$. Since $\beta^{10}$ is not divisible by $\alpha^2$, then we have
\begin{equation}
[(204n-6)\tau \beta^2+(660n-230)s_0-660r_0]=0.\label{o2}
\end{equation}
By Lemma \ref{ChSh}, we always have $s_j=0$. Then (\ref{o2}), reduces to following
\begin{equation}
(204n-6)\tau \beta^2-660r_0=0.\label{o5}
\end{equation}
Thus
\be
2(204n-6)\tau b_i\beta-660\tau b_i=0.\label{o6}
\ee
By multiplying (\ref{o6}) with $b^i$, we have 
\[
\tau=0.
\]
Thus by (\ref{a12}),  we get  $\eta=0$ and then ${\bf S}=(n+1)cF$. By (\ref{a11}),  we get $r_{ij}=0$.  Therefore Lemma \ref{ChSh}, implies that  ${\bf S}=0$.  This completes the proof.
\end{proof}

\bigskip

\noindent
{\bf Proof of Theorem \ref{MAINTHM}:} Let $F$ be an isotropic Berwald metric (\ref{IBC})  with almost isotropic flag curvature (\ref{AIFC}). In \cite{TR}, it is proved that every isotropic Berwald metric (\ref{IBC})  has isotropic S-curvature (\ref{ISC}).

Conversely, suppose that $F$ is of isotropic S-curvature  (\ref{ISC}) with scalar flag curvature ${\bf K}$. In \cite{NST}, it is showed that every Finsler metric of  isotropic S-curvature  (\ref{ISC}) has almost isotropic flag curvature (\ref{AIFC}). Now, we are going to prove that $F$ is a isotropic Berwald metric. In \cite{CS}, it is proved that $F$ is an isotropic Berwald metric (\ref{IBC}) if and only if it is a Douglas metric with isotropic mean Berwald curvature (\ref{IMBC}). On the other hand, every Finsler metric of  isotropic S-curvature (\ref{ISC}) has  isotropic mean Berwald curvature (\ref{IMBC}). Thus for completing the proof, we must  show that $F$ is a Douglas metric. By Proposition \ref{prop1}, we have ${\bf S}=0$. Therefore by Theorem 1.1 in \cite{NST},  $F$ must be of isotropic flag curvature $\boldsymbol{K}=\sigma(x)$. By Proposition \ref{prop1}, $\beta$ is  a Killing $1$-form with constant length with respect to $\alpha$, that is, $r_{ij}=s_j=0$. Then (\ref{b2}), (\ref{b3}) and (\ref{b4}) reduce to
\be
G^i-\bar{G}^i=\alpha Qs^i_{\ 0}, \ \ \ J_i=-\frac{\Phi s_{i0}}{2\alpha \Delta}, \ \  \ \bar{J}=0.\label{o7}
\ee
By (\ref{bb}), we get
\be
I_ib^i=\frac{-\Phi}{2\Delta F}(\phi -s\phi')(b^2-s^2).\label{o8}
\ee
We consider two case:\\\\
{\bf Case 1.} Let $dim M\geq 3$. In this case, by Schur Lemma $F$ has constant flag curvature and  (\ref{b5}) holds.  Thus by (\ref{o7}) and (\ref{o8}),  the equation (\ref{b5}) reduces to following
\begin{equation}
\frac{\Phi s_{i0}}{2\alpha \Delta}a^{ik}s_{k0}+\frac{\Phi s_{l0}}{2\alpha \Delta}\Big(\alpha Qs^l_{\ 0}\Big)_{.i}\ b^i-\boldsymbol{K}F\frac{\Phi}{2\Delta}(\phi -s\phi')(b^2-s^2)=0.\label{o3}
\end{equation}
By assumption $\Phi \neq 0$. Thus by (\ref{o3}), we get
\begin{equation}
s_{i0}s^i_{\ 0}+s_{l0}\Big(\alpha Qs^l_{\ 0}\Big)_{.i}\ b^i-\boldsymbol{K}F\alpha(\phi -s\phi')(b^2-s^2)=0.\label{o4}
\end{equation}
The following holds
\[
\Big(\alpha Qs^l_{\ 0}\Big)_{.i}\ b^i=sQs^i_{\ 0}+Q's^i_{\ 0}(b^2-s^2).
\]
Then (\ref{o4}) can be rewritten as follows
\begin{equation}\label{a14}
s_{i0}s^i_{\ 0}\Delta -\boldsymbol{K}\alpha^2 \phi(\phi -s\phi')(b^2-s^2)=0.
\end{equation}
By (\ref{b1}), (\ref{a5}) and (\ref{a14}), we obtain
\begin{eqnarray}
\nonumber \!\!\!\!\!\!&&\!\!\!\!\!\!\ \Big[1- \frac{s(1+2s+3s^2)}{(-1+s^2+2s^3)}+\frac{2(b^2-s^2)(1+3s)(1+s+s^2+s^3)}{(-1+s^2+2s^3)^2}\Big]s_{i0}s^i_{\ 0}\\
\!\!\!\!\!\!\!\!\!\!&-&\!\!\!\!\!\!\ \boldsymbol{K}(1+s+s^2+s^3)\alpha^2\big[1+s+s^2+s^3-s(1+2s+3s^2)\big](b^2-s^2)=0.\label{a16}
\end{eqnarray}
Multiplying (\ref{a16}) with $(-1+s^2+2s^3)^2\alpha^{12}$ yields
\[
A+\alpha B=0,
\]
where
\begin{align}
\nonumber A&=-\boldsymbol{K}b^2\alpha^{14}+(2b^2+1)(\boldsymbol{K}\beta^2 + s_{i0}s^i_{\ 0}) \alpha^{12}\\
\nonumber &+2(3 \boldsymbol{K} b^2 \beta^2 +4 s_{i0}s^i_{\ 0} b^2 - \boldsymbol{K} \beta^2- s_{i0}s^i_{\ 0} )\beta^2 \alpha^{10}\\
\nonumber &-(6\boldsymbol{K}\beta^2+11s_{i0}s^i_{\ 0} +20\boldsymbol{K}\beta^2 b^2-6s_{i0}s^i_{\ 0}) \beta^4 \alpha^8\\
\nonumber &-(-20\boldsymbol{K}\beta^2 +5\boldsymbol{K}\beta^2 b^2+8s_{i0}s^i_{\ 0})\beta^6 \alpha^6 \nonumber+(\boldsymbol{K}\beta^{10})(26b^2+5)\alpha^4 \\
&-2\boldsymbol{K}\beta^{12}(13-4b^2)\alpha^2 -8\boldsymbol{K}\beta^{14},\label{a15}\\
\nonumber B&=-(\boldsymbol{K}b^2\beta)\alpha^{12}+(1+8b^2)(\boldsymbol{K}b^2\beta^2 +s_{i0}s^i_{\ 0})\beta\alpha^{10}\\
\nonumber &-2(3 \boldsymbol{K} b^2 \beta^2 -4 s_{i0}s^i_{\ 0} b^2 +4 \boldsymbol{K} \beta^2+5 s_{i0}s^i_{\ 0} )\beta^3 \alpha^8\\
\nonumber &+(6\boldsymbol{K}\beta^2-11s_{i0}s^i_{\ 0} -20\boldsymbol{K}\beta^2 b^2) \beta^5 \alpha^6 \\
\nonumber &+(5\boldsymbol{K}\beta^9)(3b^2+4)\alpha^4+(5\boldsymbol{K}\beta^{11})(-3+4b^2)\alpha^2-20\alpha\boldsymbol{K}\beta^{13}.
\end{align}
Obviously, we have $A=0$ and $B=0$.

By $A=0$ and the fact that $\beta^{14}$ is not divisible by $\alpha^2$, we get  ${\bf K}=0$. Therefore (\ref{a16}) reduces to following
\[
s_{i0}s^i_{\ 0 }= a_{ij}s^j_{\ 0}s^i_{\ 0} = 0.
\]
Because of positive-definiteness of the Riemannian metric $\alpha$, we have $s^i_{\ 0} = 0$, i.e., $\beta$ is closed. By $r_{00}=0$ and $s_0=0$, it follows that $\beta$ is parallel with respect to $\alpha$. Then $F=\alpha+\beta+\frac{\beta^2}{\alpha}+\frac{\beta^3}{\alpha^2}$ is a Berwald metric. Hence $F$ must be locally Minkowskian.\\\\
{\bf Case 2.} Let $dim\ M=2$.  Suppose that $F$ has  isotropic Berwald curvature (\ref{IBC}). In \cite{TR}, it is proved that every isotropic Berwald metric (\ref{IBC})  has isotropic S-curvature,   ${\bf S}=(n+1)cF$. By Proposition \ref{prop1}, $c=0$. Then by (\ref{IBC}),  $F$ reduces to a Berwald metric. Since $F$ is non-Riemannian, then by Szab\'{o}'s rigidity Theorem for Berwald surface (see \cite{BCS} page 278), $F$ must be locally Minkowskian.
\qed

\bigskip

\bigskip
\noindent
{\bf Acknowledgment.}  The authors would like to thank the referee's valuable suggestions.


\noindent
Akbar Tayebi and  Tayebeh Tabatabaeifar\\
Faculty  of Science, Department of Mathematics\\
University of Qom \\
Qom. Iran\\
Email:\ akbar.tayebi@gmail.com\\
Email:\  t.tabaee@gmail.com

\bigskip

\noindent
Esmaeil Peyghan\\
Faculty  of Science, Department of Mathematics\\
Arak University 38156-8-8349\\
Arak. Iran\\
Email: epeyghan@gmail.com


\begin{thebibliography}{MaHo}
\bibitem{BM} S.  B\'{a}cs\'{o} and M. Matsumoto, {\it On Finsler spaces of Douglas type, A generalization of notion of Berwald space}, Publ. Math. Debrecen. {\bf 51}(1997), 385-406.
 \bibitem{BCS} D. Bao, S. S. Chern and Z. Shen, {\it An Introduction to Riemann-Finsler Geometry},  Springer-Verlag, 2000.
\bibitem{CC} X. Chun-Huan and X. Cheng, {\it On a class of weakly-Berwald $(\alpha,\beta)$-metrics},  J. Math. Res. Expos. {\bf 29}(2009), 227-236.
\bibitem{ChSh2} X. Cheng and Z. Shen, {\it Randers metric with special curvature properties},  Osaka. J. Math.  {\bf 40}(2003), 87-101.
\bibitem{ChSh3} X. Cheng and Z. Shen, {\it A class of Finsler metrics with isotropic S-curvature}, Israel. J. Math. {\bf 169}(2009), 317-340.
\bibitem{CS} X. Chen and Z. Shen, {\it  On Douglas metrics},  Publ. Math. Debrecen.  \textbf{66}(2005), 503-512.
\bibitem{ChWW} X. Cheng, H. Wang and M. Wang, {\it $(\alpha,\beta)$-metrics with relatively isotropic mean Landsberg curvature}, Publ. Math. Debrecen. {\bf 72}(2008), 475-485.
\bibitem{Cui} N. Cui, {\it On the S-curvature of some $(\alpha,\beta)$-metrics}, Acta. Math. Scientia, Series: A. {\bf 26}(7) (2006), 1047-1056.
\bibitem{LL} I.Y. Lee  and  M.H. Lee, {\it On weakly-Berwald spaces of special $(\alpha,\beta)$-metrics}, Bull. Korean Math. Soc. {\bf 43}(2) (2006), 425-441.
 \bibitem{Mat5} M. Matsumoto,  {\it Theory of Finsler spaces with $(\alpha, \beta)$-metric}, Rep. Math. Phys. {\bf 31}(1992), 43-84.
\bibitem{NST} B. Najafi, Z. Shen and A. Tayebi , {\it Finsler metrics of scalar flag curvature with special non-Riemannian curvature properties}, Geom. Dedicata. {\bf 131}(2008), 87-97.
\bibitem{PC1} H. S. Park and E. S. Choi,  {\it On a Finsler spaces with a special  $(\alpha, \beta)$-metric}, Tensor, N. S. {\bf 56}(1995), 142-148.
\bibitem{PC2} H. S. Park and E. S. Choi,  {\it Finsler spaces with an approximate Matsumoto metric of Douglas type},
Comm. Korean. Math. Soc. {\bf 14}(1999), 535-544.
\bibitem{PC3} H. S. Park and E. S. Choi,  {\it Finsler spaces with the second approximate Matsumoto metric}, Bull. Korean. Math. Soc. {\bf 39}(1)  (2002),  153-163.
\bibitem{PLP} H. S. Park, I. Y. Lee and C. K. Park,  {\it Finsler space with the general approximate
Matsumoto metric}, Indian J. Pure. Appl. Math. {\bf 34}(1) (2002), 59-77.
\bibitem{PLPK} H. S. Park, I.Y. Lee, H. Y. Park and B. D. Kim, {\it Projectively flat Finsler space with an approximate Matsumoto metric}, Comm. Korean. Math. Soc. {\bf 18}(2003), 501-513.
\bibitem{ShDiff} Z. Shen, {\it Differential Geometry of Spray and Finsler Spaces}, Kluwer Academic Publishers,  Dordrecht,  2001.
\bibitem{TN} A. Tayebi and B. Najafi, {\it  On isotropic Berwald metrics}, Ann. Polon. Math. {\bf 103}(2012), 109-121.
\bibitem{TPS} A. Tayebi,  E. Peyghan and H. Sadeghi, {\it On Matsumoto-type Finsler metrics}, Nonlinear Analysis: RWA. {\bf 13}(2012), 2556-2561.
\bibitem{TR} A. Tayebi and   M. Rafie Rad, {\it S-curvature of isotropic Berwald metrics}, Sci. China. Series A: Math.  {\bf 51}(2008), 2198-2204.
\end{thebibliography}
\end{document}